\documentclass{article}

\usepackage{microtype}
\usepackage{graphicx}
\usepackage{subcaption}
\usepackage{booktabs} 

\usepackage{hyperref}

\usepackage[preprint]{icml2026}

\usepackage{amsmath}
\usepackage{amssymb}
\usepackage{mathtools}
\usepackage{amsthm}
\usepackage{tikz-cd}
\usepackage{bm}

\usepackage[capitalize,noabbrev]{cleveref}

\theoremstyle{plain}
\newtheorem{theorem}{Theorem}[section]
\newtheorem{proposition}[theorem]{Proposition}
\newtheorem{lemma}[theorem]{Lemma}
\newtheorem{corollary}[theorem]{Corollary}
\newtheorem{conjecture}[theorem]{Conjecture}
\theoremstyle{definition}
\newtheorem{definition}[theorem]{Definition}

\theoremstyle{remark}
\newtheorem{remark}[theorem]{Remark}

\usepackage[textsize=tiny]{todonotes}

\icmltitlerunning{Constructive Bounds for Topological Expressivity}

\begin{document}

\twocolumn[
  \icmltitle{Constructive Bounds for the Topological Expressivity of Neural Networks} 



  \icmlsetsymbol{equal}{*}

  \begin{icmlauthorlist}
    \icmlauthor{Kosio Beshkov}{yyy}
  \end{icmlauthorlist}

  \icmlaffiliation{yyy}{Department of Physics, University of Oslo, Oslo, Norway}

  \icmlcorrespondingauthor{Kosio Beshkov}{kosio.neuro@gmail.com}

  \icmlkeywords{Machine Learning, ICML}

  \vskip 0.3in
]



\printAffiliationsAndNotice{}  

\begin{abstract}
  The expressivity of a neural network $\Phi$ describes the class of functions that a network can express. The term \textit{topological expressivity} refers to a family of measures particularly suited to understand binary classification and is often formulated in terms of the Betti numbers of a decision boundary $\beta (\Phi^{-1}(0))$. Despite its popularity, this measure is not as useful for regression and other problems that do not involve a decision boundary. Instead, we propose to study topological expressivity in terms of the Betti numbers of the output of a neural network $\beta(\Phi(X))$. We formally prove that a ReLU network with architecture $(d,\sum\limits_{k=1}^{d
  } \beta_k(3k+2),\sum\limits_{k=1}^{d}\beta_k(k+1))$ can express spaces homotopy equivalent to wedge sums of spheres with arbitrary Betti numbers. This gives upper bounds for any arbitrary connected ($\beta_0=1$) sequence of Betti numbers, but by constructing spaces homotopy equivalent to products of $N$ spheres of dimension $d_1,...,d_N$ we achieve much better bounds $(\sum\limits_{n=1} d_n, 3\sum\limits_{n=1}\limits^N d_n + 2N, \sum\limits_{n=1}\limits^N d_n + N) $ for Betti numbers given by $\prod_n (1+x^{d_n})$. Finally we conjecture a much stronger tight bound which gives the architectural bound $(d,d+1,d+1)$ for a space homeomorphic to the sphere.
\end{abstract}

\section{Introduction}
    One of the fundamental problems in machine learning is to determine what model one needs to perform a given task or to learn a given function. Once such a model is selected it fixes a corresponding class of functions that it can implement. The nature of this class is usually described by a vast number of different measures of model expressivity.

    In the context of neural networks there are several universal approximation theorems \cite{cybenko1989approximation, hornik_approximation_1991, arora2018understanding} that guarantee that there is some network that will perform the task we are interested in. What then remains to be answered is not if such a network exists, but rather which networks are capable of completing the task and how does the class of functions that a network can express depend on its architecture. There are several measures of expressivity that attempt to illuminate this problem. Some essential examples are Vapnik-Chervonenkis dimension \cite{baum1988size}, Rademacher complexity \cite{bartlett2002rademacher}, the number of linear regions in networks with piecewise-linear (PWL) activations \cite{montufar2014number}, the length of the output trajectory \cite{raghu2017expressive} and the topological complexity of decision boundaries \cite{bianchini_complexity_2014}.

     The last example of expressivity, namely that of topological expressivity, is particularly relevant for the present work. Given a neural network $\Phi:\mathbb{R}^d\to\mathbb{R}$ its topological expressivity is defined as $\beta(\Phi^{-1}([0,\infty)))$, where $\beta$ denotes the vector of Betti numbers and $\Phi^{-1}([0,\infty))$ is one of two decision regions. Conceptually, this measures the complexity of the decision boundary separating two classes in the context of binary classification. 

    Several papers have pushed this idea further. Two examples in which the authors empirically explore which architectures are capable of expressing particular decision boundaries are \cite{guss_characterizing_2018, petri_topological_2020}. On the more theoretical side, \cite{grigsby_transversality_2021} prove that networks with one hidden layer of dimension $d+1$ have decision regions with at most one bounded connected component and \cite{ergen2024topological} constructively derive both upper and lower bounds on the Betti numbers of the decision regions. 

    While all measures of expressivity have their advantages, it is important to highlight that most of them are especially suited for classification problems and are therefore harder to understand when working with regression tasks. This is clearly the case when looking at the topology of the decision boundary as in problems that are not about classification, the notion of a decision boundary is undefined. Furthermore, as illustrated by \cite{petri_topological_2020}, it is possible to have data that is highly topologically complex but still separable by a hyperplane (imagine two tori separated by a hyperplane). This means that this notion of topological complexity ignores the topology of the data and is not applicable to our understanding of the plethora of topological transformations that a network needs to perform in non-classification tasks.

    A different set of works studies the topology of neural representations, by which we mean $\beta(\Phi(X))$, where $X$ is some dataset. For example \citet{naitzat2020topology} use persistent homology \cite{carlsson_topology_2009} to look at how topology evolves through the hidden layers of a network, \citet{wheeler2021activation} use persistence landscapes to explore the same question, \citet{beshkov2025relative} uses quotient homology to study how neural networks can glue data and generate non-trivial topologies and \citet{ren2026low} studies how low-dimensional topological invariants are transformed through the layers of networks. We will use the term \textit{representational topological expressivity} to talk about the Betti numbers of the image of a network $\beta(\Phi(X))$. This measure takes both the data as well as the task into account. It therefore aligns much better with non-classification tasks which are abundant in many scientific fields. Let us list some examples of such tasks to highlight its importance.
    \begin{itemize}
        \item \textbf{Regression tasks} require one to find a map $f:X\to \mathbb{R}^m$ and are used in protein structure prediction \cite{abramson2024accurate}, weather forecasting \cite{hewage2021deep}, predicting neural population activity \cite{yamins2014performance} and 3D pose estimation \cite{zheng2023deep}.
        \item \textbf{Density estimation tasks} require one to fit a distribution using a map $f:X \to \mathbb{R}^+$. This type of tasks is the main motivation behind methods such as normalizing flows \cite{papamakarios2021normalizing}, latent densities in variational autoencoders \cite{kingma2013auto} and some diffusion models \cite{gao2023implicit}.
        \item \textbf{Representation learning} requires one to optimize the representation of data so that it can be used to solve other tasks more efficiently. Examples of this are learning good word embeddings \cite{mikolov2013efficient}, protein embeddings \cite{hayes2025simulating} and finding low dimensional latent spaces for the generation of neural activity \cite{pei2021neural}.
        \item \textbf{Differential equation tasks} require that one find a way to map states to a function. Examples of this are physics informed neural networks \cite{raissi2017physics}, neural ODEs \cite{chen2018neural} and neural operators \cite{li2020fourier}. 
    \end{itemize} 

 This is not nearly an exhaustive list and it should be clear that our current notion of topological expressivity cannot be applied in the context of these tasks. Therefore, it is important to extend our theory to match the needs of machine learning practitioners.

    \subsection{Our contributions}
    The previously mentioned studies of the topology of neural representations are empirical or methodological in nature and do not explicitly relate their results to any notion of expressivity of a network. In this paper we provide, what is to our knowledge, the first architectural bounds on the topology of the representations of ReLU neural networks. We attempt to classify networks by the Betti numbers that they can generate given that the data comes from some contractible compact set in $\mathbb{R}^d$. Instead of evaluating the Betti numbers for a network with a given architecture, we approach the problem in reverse and provide upper bounds on the size of the architecture that can generate particular Betti numbers.

    We construct a set of PWL functions whose image is homotopy equivalent to that of a $d$-sphere $S^d$ and then construct a neural network that implements exactly these functions. Following this, we show how one can compose such networks to generate spaces that are homotopy equivalent to wedge sums and products of spheres. Finally, we conjecture a much stronger tight bound that provides the optimal architecture for expressing spaces homeomorphic to spheres.

\section{Results}

Before we begin with proving our main results, let us outline our assumptions along with the notation that we use. We will often interchangeably speak about the image of a map $f:\mathbb{R}^d\to \mathbb{R}^{d+1}$ and a restricted version of this map $f:X\to \mathbb{R}^{d+1}$ where $X$ is some compact contractible set. Usually this should not cause confusion, as our constructions either have compact supports or $f(\mathbb{R}^d) = f(X)$ for a compact set $X$. Unless otherwise stated we will work with reduced singular homology groups with coefficients in $\mathbb{Z}$, which we will briefly write as $\tilde{H}_n(X)$. We use the symbol $\simeq$ to denote homotopy equivalence and $\cong$ for homeomorphism or isomorphism which should be clear from context. To avoid disturbing the flow of the argumentation we have moved the longest proofs to the Appendix \ref{Appendix} and only provide proof sketches in the main text.

To find an upper bound for a neural architecture that can express a particular combination of Betti numbers we take the following approach. Using the fact that every ReLU network is a PWL map \cite{arora2018understanding} we start with an explicit construction of a PWL map that provably maps a contractible set in $\mathbb{R}^d$ to a space homotopy equivalent to $S^d$. We prove that this map can be implemented by a neural network with architecture $(d,W_d,d+1)$, with $W_d=\min(2d(d+1),3d+2)$. This gives us a network expressing Betti number $\beta_n=1$ for $n=d$ and 0 otherwise. By repeatedly applying several such maps to different subsets of $\mathbb{R}^d$, we are able to achieve arbitrary $\beta=(0,\beta_1,...,\beta_d)$. This repeated application amounts to repeatedly adding subnetworks from which we get the architecture $(d,\sum\limits_{k=1}^{d
  } \beta_kW_k,\sum\limits_{k=1}^{d}\beta_k(k+1))$ which is the upper bound in question. Alternatively, by expanding the input domain and stacking $N$ independent networks we get architecture $(\sum\limits_{n=1}\limits^N d_n, \sum\limits_{n=1}\limits^N W_{d_n}, \sum\limits_{n=1}\limits^N d_n + N)$, with $W_{d_n} = \min(2d_n(d_n+1),3d_n+2)$.

\subsection{Expressing the $d$-sphere}
In the first part we show how a network can be used to express $S^d$. We begin by introducing some helpful functions that can be combined to create the PWL map whose image is a space homotopy equivalent to $S^d$.

\begin{definition}
    The \textit{hat function} $p_a:\mathbb{R}^{d}\to\mathbb{R}$ centered at $a=(a_1,...,a_d)$ is given by the expression,
    \begin{equation*}
        p_a(x_1,...,x_d) = \max(0,1-\sum\limits_{l=1}^{d}|x_l-a_l|).
    \end{equation*}
\end{definition}

\begin{figure*}
    \centering
    \includegraphics[width=1\linewidth]{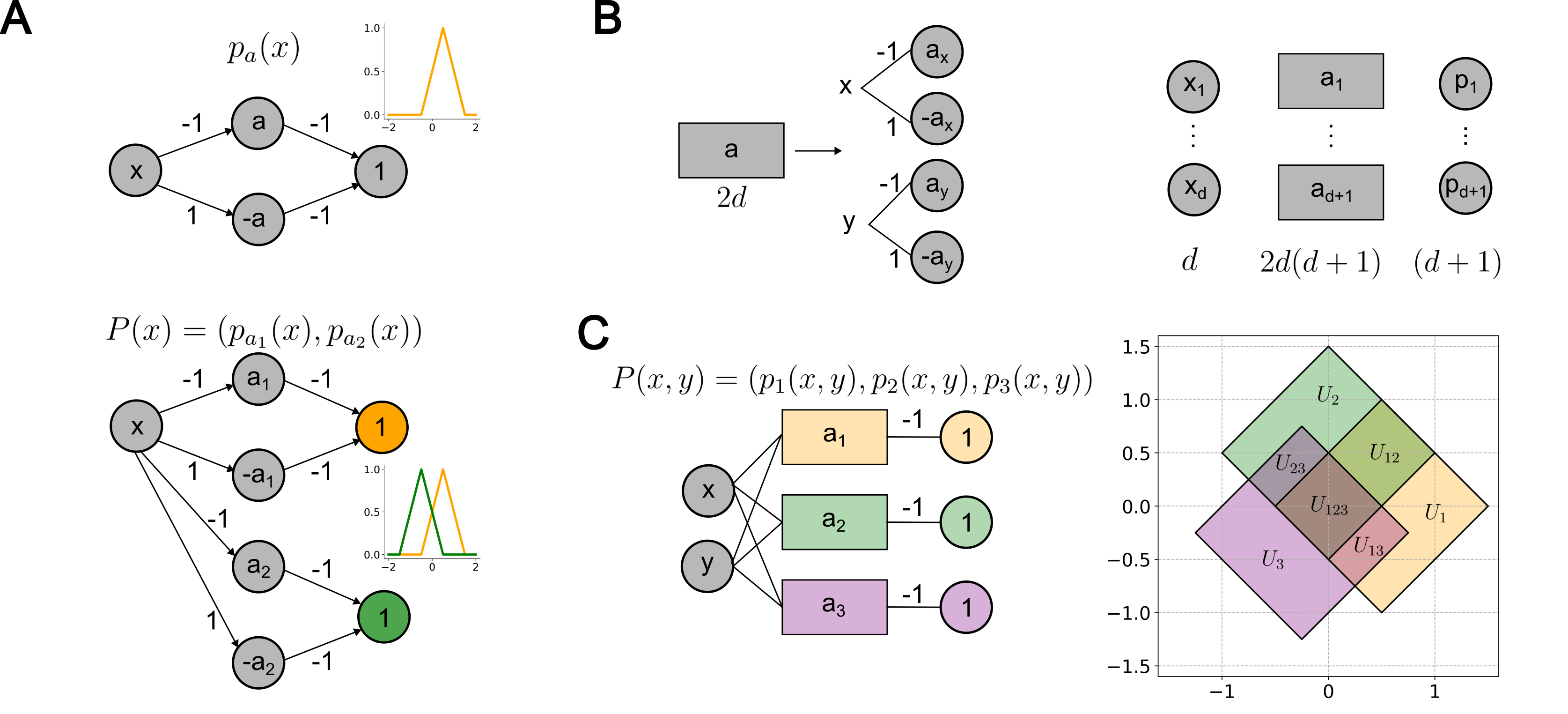}
    \caption{An illustration of how hat functions can be expressed by a neural network. \\
    \textbf{A.} A diagram of the ReLU neural network that expresses a single hat function in one dimension. Below it we show how one can stack several hat functions by varying the biases in the middle layer and adding additional output neurons. \\
    \textbf{B.} We use a rectangle to express the subnetwork in the middle layer that implements one absolute value function per input neuron. The arrow points to an example of such a layer in two dimensions. To the right we show how such a network would look like in $d$ dimensions. \\
    \textbf{C.} The corresponding network in two dimensions along with the Helly-arrangement described in Proposition \ref{Helly construction}.}
    \label{fig:construction}
\end{figure*}
Next we show that such functions can be easily represented by a neural network.

\begin{proposition}
    A single hat function $p_a$ can be realized by a neural network with an architecture $(d,2d,1)$. Additionally, after setting $x=(x_1,...,x_d)$, the ordered set of hat functions $P(x) = (p_1(x),...p_{d+1}(x))$ (we will from now on index the order rather than the center $a$) can be realized by a neural network architecture $(d,2d(d+1),d+1)$ or $(d,3d+2,d+1)$. 
\end{proposition}

\begin{proof}
    The proof of this is constructive. The absolute value in $p_k(x)$ on each coordinate $x_k$ can also be written as $|x_k-a_k|=\max(0,x_k-a_k)+\max(0,a_k-x_k)$. This corresponds to two ReLU neurons with weights $\{1,-1\}$ and biases $\{-a_k,a_k\}$. This pattern repeats $d$ times, once for every input coordinate. After setting all weights from the second to the third layer to $-1$ and setting the bias of the last neuron to 1 one gets the same expression as the definition of $p_k(x)$. Since there are $d$ input dimensions and 2 neurons are required to express the absolute value, there need to be $2d$ neurons in the second layer. Therefore the network will have $(d,2d,1)$ architecture.

    The construction of $P(x)$ follows simply from stacking $d+1$ subnetworks that realize different hat functions $p_k(x)$. Since the input space is shared, the architecture becomes $(d,2d(d+1),d+1)$. A visual intuition for the proof is given in Figure \ref{fig:construction}.

    For the $(d,3d+2,d+1)$ bound, let $a_{k,l}$ be the $l$-th entry of the $k$-th center. Then $|x_l-a_{k,l}| = 2\max(0,x_l-a_{k,l})-x_l+a_{k,l}$ and $\sum x_l = \max(0, \sum x_l)+\max(0,-\sum x_l)$. So we write $2\sum_l \max(0,x_l-a_{k,l})-\max(0, \sum_l x_k)+\max(0,-\sum_l x_k)+\sum_l a_{k,l}=\sum_l|x_l-a_{k,l}|$. From this we have an expression for each of the $d+1$ functions in terms of ReLU activations. The constructed centers have three values $\{0,\frac{1}{2},-\frac{1}{2d}\}$, so we need $3d$ neurons realizing them the extra two terms $\max(0,\sum x_l)$ and $\max(0, -\sum x_l)$ add two more neurons to the width giving the architecture $(d,3d+2,d+1)$. 
\end{proof}

Next we have to show that this image of this function is in fact homotopy equivalent to $S^d$. The support of a hat function, which we denote by $U = \text{supp}(p_a) = \{x\in\mathbb{R}^d: p_a(x) > 0\}$, is an $L_1$ ball, which is an open convex set with the shape of a cross-polytope. Consider the following arrangement of $d+1$ such hat functions.

\begin{definition}
    Given an ordered set of hat functions $P = (p_1,p_2,...,p_{d+1})$ with supports $\{U_1,...,U_{d+1}\}$, we will say that $\Omega = \bigcup\limits_{k=1}^{d+1}U_k$ is in a \textit{Helly-arrangement} if $\bigcap\limits_{k=1}^{d+1}U_k \neq \emptyset$.
\end{definition}
The Helly arrangement property guarantees that, if we define $U_J = \bigcap\limits_{j\in J}U_j$, then for any $J \subset \{1,...,d+1\}=I$ we also have $U_I\subset U_J$ and $U_J\neq \emptyset$ for all $J$. Furthermore, the fact that every $U_i$ is convex implies that $U_J$ is convex and there is a retraction to a point $x_0\in U_I$. Therefore $\Omega$ is star-shaped and the supports form a good cover for it. Such an arrangement is shown in panel C of Figure \ref{fig:construction} and can be constructed by the following choice of centers for the hat functions.

\begin{proposition}
\label{Helly construction}
    The open cover $\Omega= \bigcup\limits_{k=1}^{d+1} U_k$ generated by the supports of $d+1$ hat functions with centers $a_k=\frac{1}{2}e_k$ (where $e_k$ is the standard basis) for $k\neq d+1$ and $a_{d+1} = -\frac{1}{2d}\sum_le_l$ forms a Helly-arrangement.
\end{proposition}

\begin{proof}
    We only need to confirm that the maximal intersection is not empty. If we plug the zero vector into the collection of hat functions $P$, we get $p_i(0) = \max(0,1-\sum|0-\frac{1}{2}e_i|)=\frac{1}{2}$ for all $i<d+1$ and $p_{d+1}(0) = \max(0,1-d/{2d}) = \frac{1}{2}$. Therefore the zero vector is in the support of all the functions and the maximal intersection is non-empty.
\end{proof}

The fact that $\Omega$ is star-shaped implies that $\tilde{H}_k(\Omega)=0$ for all $k$, but this can also be seen by the nerve of the cover $N(\Omega)=\{J\subseteq I:\bigcap\limits_{i\in J} U_j\neq \emptyset \}$ which in this case is a $d$-simplex and therefore has trivial homology groups (see \cite{hatcher2005algebraic} for more information). The interesting question is to understand what happens to $\Omega$ after it is mapped by the PWL map $P$. For this we prove the following lemma.

\begin{lemma}
\label{semilattice lemma}
    Given a cover of $\Omega$ defined by the supports of $P$, the intersections can be expressed as $\bigcap\limits_{j\in J} P(U_j) = P(\bigcap\limits_{j\in J} U_j)$.
\end{lemma}

\begin{proof}
    Every set $U_i$ corresponds to an $L_1$ ball and $P(U_i) = \{P(x):p_i(x)>0\}$. Then looking at an arbitrary intersection $U_{ij} = U_i\cap U_j = \{x: p_i(x)>0 \land p_j(x)>0\}$, we can also write $P(U_i\cap U_j) = \{P(x):p_i(x)>0 \land p_j(x)>0\}$. What about $P(U_i)\cap P(U_j)$? We can write out the sets,
    \begin{equation*}
    \begin{split}
        &P(U_i)\cap P(U_j) =\\
        &\{P(x):p_i(x)>0\}\cap\{P(x):p_j(x)>0\} = \\
        &\{P(x):p_i(x)>0 \land p_j(x)>0\} =\\
        &P(U_i\cap U_j).
    \end{split}
    \end{equation*}
    Since we can do this for arbitrary intersections this concludes the proof.
\end{proof}

It is clear that using the supports of $P$ to define a cover on $\Omega$ leads to some nice relation to its the nerve. It turns out that the two nerves are actually isomorphic.

\begin{theorem}
    \label{nerve homotopy}
    Given a good cover in Helly arrangement $\Omega = \bigcup\limits_{k=1}^{d+1}U_k$ and a map $P:\Omega\to P(\Omega)$ such that $\bigcap\limits_{j\in J} P(U_j) = P(\bigcap\limits_{j\in J} U_j)$, there is an isomorphism between $N(\Omega)$ and $N(P(\Omega))$ as sets. 
\end{theorem}

\begin{proof}
    This amounts to showing that $J\in N(\Omega) \iff J\in N(P(\Omega))$.

    The first direction is straightforward, if $J \in N(\Omega)$, then $P(U_J) \neq \emptyset$, as $J$ hat functions are supported on this set.

    For the second direction, $J\in N(P(\Omega))$ implies that $\bigcap\limits_{j\in J} P(U_j) \neq \emptyset$ and by applying Lemma \ref{semilattice lemma} we can write it as $P(\bigcap\limits_{j\in J}U_j)\neq \emptyset$, which can only be true if $\bigcap\limits_{j\in J}U_j$ is non-empty. $U_J$ is non-empty for any $J$ since $U_I \subset U_J$ and $U_I\neq \emptyset$.

    Therefore $N(\Omega) \cong N(P(\Omega))$ as sets.
\end{proof}

If we extend $P$ to $\hat{P}:\mathbb{R}^d\to\mathbb{R}^{d+1}$ we have $\hat{P}(\mathbb{R}^d-\Omega) = 0 = \hat{P}(\partial \Omega)$. Then we can prove the following theorem.

\begin{theorem}
    Consider $X=\Omega\cup\partial \Omega$, then $\hat{P}(X)$ has a good cover $C=(\bigcup\limits_{i \in I}\bar{R}_{I\setminus i})\bigcup\bar{V}_I$, where $\bar{R}_J = \{y\in \hat{P}(X): y_j=0 \forall j \notin J\}$ and $\bar{V}_I = \hat{P}(\bar{U}_I)$. This implies that $\hat{P}(X) \cong S^d$ and therefore $\tilde{H}_*(\hat{P}(X)) \cong \tilde{H}_*(S^d)$
\end{theorem}

\begin{proof}
\textit{sketch:}
\\
The full proof of this statement can be found in Appendix \ref{Appendix}. The intuition is that since $\Omega$ is starshaped, then $X=\bar{\Omega}$ is homeomorphic to the disk $D^d$, additionally we know that $\hat{P}(\partial \Omega)=0$ glues its boundary to a point. This means that $\hat{P}$ acts similarly to a quotient map on the boundary of $D^d$ which gives $D^d/S^{d-1} \cong S^d$. What remains to be shown is that $\hat{P}$ does not lead to non-homotopic identifications elsewhere on $X$.  

To do that we take proposed cover and prove that all possible intersections of its subsets are contractible. Then by Borsuk's nerve theorem \cite{borsuk1948imbedding} we show that the nerve of this cover has the structure of the boundary of the simplex $\partial \Delta^{d+1}\cong S^d$. This simplex is exactly homotopy equivalent to $S^d$. 
\end{proof}

\subsection{Extending the construction to wedge sums of spheres}

\begin{figure*}
    \centering
    \includegraphics[width=1\linewidth]{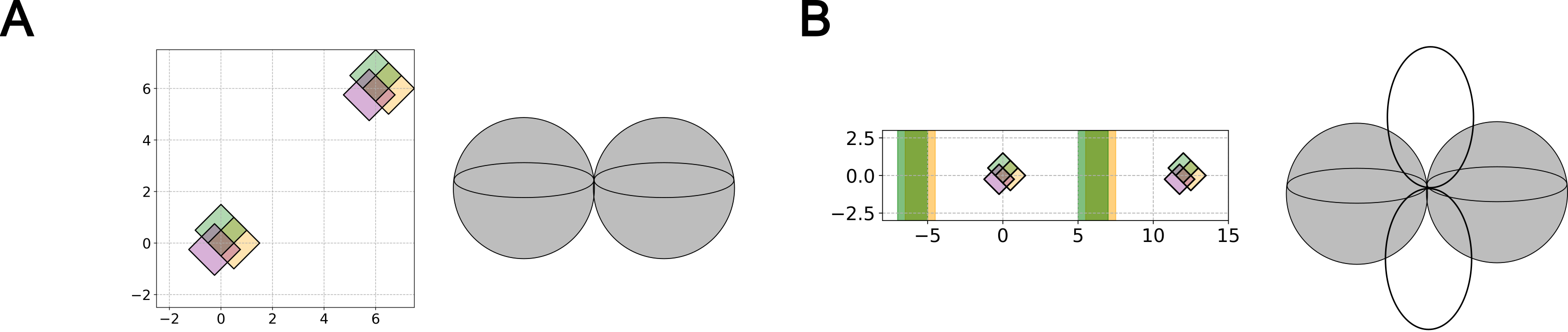}
    \caption{\textbf{A.} An example of the wedge of two spheres that arises through the Helly-arrangement living within $[-1.5,1.5]^2$ and $[4.5,7.5]^2$. \textbf{B.} An example of the wedge of two spheres and two circles along with an arrangement that defines it. The four sets are defined strictly within $[-7.5,-4.5]\times(-\infty,\infty)$, $[-1.5,1.5]^2$, $[4.5,7.5]\times(-\infty,\infty)$} and $[10.5,13.5]\times[-1.5,1.5]$ respectively. Notice that the configurations generating the circles span the whole $y$ direction, thus it is important to always choose the shifts on the $x$ axis, otherwise two circles might end up intersecting in more places than the origin.
    \label{fig:wedge}
\end{figure*}

So far we have shown that our construction is capable of representing spaces that have the homology of a sphere and therefore have Betti numbers of the form $\beta_n = 0$ for $n\neq d$ and $\beta_d = 1$. We now show that it is possible to extend our result to represent wedge sums of spheres. As a reminder the wedge sum of two spaces is defined as $X\vee Y = X\sqcup Y/(x_0\sim y_0)$, where $x_0\sim y_0$ indicates that the two spaces are glued at those two points.

First off, let us explicitly add the centers of the different hat functions in $P(x;A)$, where $A = (\bm{a}_1,...,\bm{a}_{d+1})$ and each $\bm{a}_k= (a_{k,1},...,a_{k,d})$ is a vector denoting the center of the hat function. Notice that the function $P(x;A)$ is invariant under the transformation $(x;A) \to (x+v,A+v\bm{1}^T)$. Geometrically, this means that if we shift all hat functions in the same way as we shift the inputs the output will be the same. Also notice that the support of the construction we defined in Proposition \ref{Helly construction} is strictly contained in the hypercube $[-1.5,1.5]^d$. Therefore, if we repeat the same construction outside of this hypercube, we can expect to get a new sphere in the image of the collective map $(P_1(x;A),P_2(x;,A+v\bm{1}^T))$. This means that we can tile $\mathbb{R}^d$ by different families of hat functions, which leads us to the following construction.

\begin{definition}
    Define the family of maps $\bm{P} = (\hat{P}_1,...,\hat{P}_m):\mathbb{R}^d\to\mathbb{R}^{m(d+1)}$ such that $\hat{P}_k(x) = \hat{P}(x;\hat{A}+v_k\bm{1}^T)$, where $v_k\in 6\mathbb{Z}^d$ is different for all $k$ and $\hat{A}$ is the Helly arrangement defined in Proposition \ref{Helly construction}.
\end{definition}

Every $\hat{P}_k(x,\hat{A}+v_k)$ is supported on a set that is contained in $[-1.5,1.5]^d+v_k$ and is disjoint from the supports of other such functions (this is guaranteed by picking $6\mathbb{Z}^d$). Furthermore, we already know that $\hat{P}_k(\mathbb{R}^d) \cong S^{d}$. This motivates our following theorem,

\begin{theorem}
\label{wedge theorem}
    Given the family of maps $\bm{P} = (\hat{P}_1,...,\hat{P}_m)$ acting on $\mathbb{R}^d$, we have $\tilde{H}_n(\bm{P}(\mathbb{R}^d)) \cong \tilde{H}_n(\bigvee\limits^{m} S^d)$.
\end{theorem}
\begin{proof}
    The wedge sum of two spaces is defined by $X\vee Y = X\sqcup Y/\sim$, where $\sim$ implies identifying $x_0\sim y_0$ to each other. The image of each $\hat{P}_k$ lives in a separate subspace $\mathbb{R}^{d+1}\subset \mathbb{R}^{m(d+1)}$ and we can consider the disjoint union $\bigsqcup\limits_{k=1}^{m} \hat{P}_k(\mathbb{R}^d)$. 
    
    However, every $\hat{P}_k$ projects the complement of its support to $\{0\}$, which implies that the images of all $\hat{P}_k$ intersect in the origin. Therefore, we need to identify all $\{\{0\}_1,...,\{0\}_m\}$ (which are separate elements in the disjoint union) with each other. This is equivalent to gluing all $\hat{P}_k(\mathbb{R}^d)$ at the origin, which gives $\bigvee\limits_{k=1}^{m}\hat{P}_k(\mathbb{R}^d) \cong \bigvee\limits^{m} S^d$. Therefore $\tilde{H}_n(\bm{P}) \cong \tilde{H}_n(\bigvee\limits^{m} S^d)$ follows from the homotopy invariance of homology. From the fact that the wedge sum of spaces reduces to the direct sum of their homology groups, we have $\tilde{H}_n(\bigvee\limits_{k=1}^{m}\hat{P}_k(\mathbb{R}^d)) = \bigoplus\limits_{k=1}^{m}\tilde{H}_n(\hat{P}_k(\mathbb{R}^d)) \cong \bigoplus\limits^{m}\tilde{H}_n(S^d)$ and 
    \begin{align*}
        \tilde{H}_n(\bm{P}) = 
        \begin{cases}
            \mathbb{Z}^m & n=d\\
            0 & n\neq d.
        \end{cases}
    \end{align*}
\end{proof}

This construction is more powerful than the previous one as it can express Betti numbers of the form $\beta_n=0$ for $n\neq d$ and $\beta_n=m$ for $n=d$. Since we are simply repeatedly adding the same amount of hat functions for each sphere and we project them to independent spaces, the network architecture that can express this is $(d,mW,m(d+1))$. Concrete examples of this can be seen in Figure \ref{fig:wedge}.

Finally, it is straightforward to see that one can use a subset of coordinates $x_C$ for $C \subset \{1,...,d\}$ to express spheres of lower dimension, though some care has to be taken in choosing $v_k \in 6\mathbb{Z}^d$. Assuming $1\in C$, a safe strategy is to choose $v_k = (6k,0,...0)$. In a neural network this would correspond to setting the weights between the first and the second layer to 0 for the neurons not in $C$. This brings us to our final construction, which establishes the most general bound on homology groups. 

\begin{theorem}
\label{thm: wedge sum networks}
    A neural network of architecture $(d,\sum\limits_{k=1}^d \beta_k W_k,\sum\limits_{k=1}^d \beta_k(k+1))$ can express a space of Betti numbers $\beta= (0,\beta_1,...,\beta_d)$. Additionally, for this network $\bm{P}(\mathbb{R}^d)\simeq\bigvee\limits_{k=1}^d(\bigvee\limits^{\beta_k} S^{k})$.
\end{theorem}
\begin{proof}
    Start with the $\beta_d$ full dimensional spheres. Representing each such sphere requires a network of architecture $(d,W_d,d+1)$. We can pick $\hat{P}_1(x)$ to have support over a subset of $[-1.5,1.5]^d$. We can pick the next function $\hat{P}_2(x)$ to have support over a subset $[4.5,7.5]\times[-1.5,1.5]^{d-1}$. We can repeat this for all $\beta_d$ spheres, by picking $\text{supp}( \hat{P}_{i}(x)) \subset [-1.5+6(i-1),1.5+6(i-1)]\times[-1.5,1.5]^{d-1}$ while $i\leq \beta_d$ and since the input space is shared we only need to stack the middle and the last layer of the network. Therefore we get the architecture $(d,\beta_dW_d,\beta_d(d+1))$.

    Next we represent the $\beta_{d-1}$, $d-1$-spheres. Each such sphere requires $(d-1, W_{d-1},d)$ neurons. We have to define their supports such that $\text{supp}(\hat{P}_{\beta_d+i}) \subset[-1.5+6(\beta_d+i-1),1.5+6(\beta_d+i-1)]\times[-1.5,1.5]^{d-2}$. We inductively do this for the spheres of any dimension. Since we know that $\beta_k$ copies of $k$-dimensional spheres are expressed by architectures $(k,\beta_kW_{k},\beta_k(k+1))$. Summing over all dimensions gives us the bound $(d,\sum\limits_{k=1}^d\beta_kW_{k},\sum\limits_{k=1}^d\beta_k(k+1))$.

    Finally, the fact that $\bm{P}(\mathbb{R}^d)\simeq\bigvee\limits_{k=1}^d(\bigvee\limits^{\beta_k} S^{k})$ is essentially a restatement of the result in Theorem \ref{wedge theorem}, but with the inclusion of lower-dimensional spheres. Panel B in Figure \ref{fig:wedge} provides an example of the space $S^1\vee S^1\vee S^2 \vee S^2$, which can be represented by a network with architecture (2,24,10).
\end{proof}

This result is much more powerful as it provides a bound for the network size required to express a space with any possible Betti numbers. As we will see below it is very unlikely that this is the best possible bound. However, it is worth noting that its linear relation to Bettin numbers indicates that neural networks can be an efficient way to represent complex topological spaces.

\subsection{Exponential complexity in products of spheres}

The wedge sum bound is useful to demonstrate that any sequence of Betti numbers can be expressed by a neural network. However, it constructs spaces homotopy equivalent to wedge sums of spheres, but not to products of spheres and is provably less than optimal for certain Betti number sequences. 

To see this, take the torus $T^2 = S^1\times S^1$. It has $\beta_1=2$ and $\beta_2 = 1$ which, by the wedge sum bound, would require an architecture of $(2,16,7)$. But we can do much better if we take two networks of shape $(1,4,2)$ realizing $S^1$ and stack them together to get a larger network of shape $(2,8,4)$. Given the following lemma, this is a torus realized with less total neurons.

\begin{lemma}
\label{lemma:product}
    Let $\Phi_1:X\to Y$ and $\Phi_2: X'\to Y'$ be maps realized by networks of shape $(a,b,c)$ and $(a',b',c')$ respectively. The product map $\Phi_1\times \Phi_2$ is realized by a network with architecture $(a+a',b+b',c+c')$ and $(\Phi_1\times \Phi_2)(X\times X') = \Phi_1(X)\times \Phi_2(X')$ .
\end{lemma}

\begin{proof}
    Simply stack the two networks without connecting them in any way. This corresponds to a block diagonal weight matrix and concatenated biases. Each block acts independently on the input, so applying the function to the product is the same as taking the product of the function applied to $X$ and $X'$.
\end{proof}

\begin{corollary}
    For an arbitrary set of target spaces $(S^{d_1},...,S^{d_N})$ realized by architectures $(d_n,W_{d_n},d_n+1)$, we simply sum each architecture to get the total architecture realizing such a space. This gives $(\sum\limits_{n=1}\limits^N d_n,\sum\limits_{n=1}\limits^N W_{d_n},\sum\limits_{n=1}\limits^N d_n+N)$.

    For $S^1\times S^1$, we have $(1,4,2)$ repeated twice so we get $(2,8,4)$.
\end{corollary}

This construction a case in which linear scaling of network size leads to exponential scaling in the Betti numbers. From the Künneth formula on torsion-free spaces \citep{hatcher2005algebraic} (which spheres are) we have,

\begin{align*}
    H_k(X\times Y)\simeq\bigoplus_{i+j=k} H_i(X)\otimes H_j(Y).
\end{align*}

This implies that $\beta_k(X\times Y) = \sum\limits_{i+j=k}\beta_i(X)\beta_j(Y)$ which by writing out the Poincaré polynomials gives the generating functions $\mathcal{P}(X\times Y) = \mathcal{P}(X)\mathcal{P}(Y)$. For spheres, these are $\mathcal{P}(S^n) = 1+x^n$ so the total generating function for products of spheres $\bigtimes\limits_{n} S^{d_n}$ is,

\begin{equation*}
    \mathcal{P}(\mathop{\bigtimes}\limits_{n} S^{d_n}) = \prod\limits_{n}{(1+x^{d_n})}.
\end{equation*}

To see how the Betti numbers are generated by such polynomials, consider the space $S^1\times S^1 \times S^2$. Its Poincaré polynomial is given by $\mathcal{P}(S^1\times S^1 \times S^2) = (1+x)^2(1+x^2) = 1+2x+2x^2+2x^3+x^4 $. By reading off the coefficients of this polynomial we get the Betti numbers $\beta = (1,2,2,2,1)$. In principle, the sum of such Betti numbers scales exponentially with the number of spheres in the product, whereas the network size scales linearly. 
\begin{theorem}
    \label{thm: exp scaling}
    $\sum_k \beta_k(\mathop{\bigtimes}\limits_{N} S^1)$ can scale exponentially as $2^N$, whereas the corresponding architecture scales linearly as $7N$.
\end{theorem}

\begin{proof}
    Take the $N$-torus given by $\bigtimes^N S^1 = T^N$, its Poincaré polynomial is given by $(1+x)^N$. The coefficients of such a polynomial are the binomial coefficients $\beta_k = C(N,k)$ and the sum of binomial coefficients is $\sum C(N,k) = 2^N$.

    The corresponding architecture for one circle is $(1,4,2)$ and by Lemma \ref{lemma:product} after stacking $N$ copies of it we get $(N,4N,2N)$ which can realize an $N$-torus. So the number of neurons scales as $7N$ which is linear.
\end{proof}

The fact that exponential growth in Betti numbers is realized by linear growth in the number of neurons shows that it is quite easy to realize spaces with high representational topological complexity. However, it is worth pointing out that this bound only works on palindromic ($\beta_n = \beta_{\sum d_n-n}$) Betti numbers. The wedge bound applies for any possible sequence of them, but also requires an exponential number of neurons since the Betti numbers appear in the bound. Finally, we state these results as a theorem analogous to Theorem \ref{thm: wedge sum networks} for products of spheres.
\begin{theorem}
    A neural network of shape $(\sum\limits_{n=1}\limits^N d_n, \sum\limits_{n=1}\limits^N W_{d_n}, \sum\limits_{n=1}\limits^N d_n + N)$ can express a space with Betti numbers given by the coefficients of the Poincaré polynomials, $\mathcal{P}(\mathop{\bigtimes}\limits_{n} S^{d_n}) = \prod\limits_n(1+x^{d_n})$. Additionally, for this network $\textbf{P}(\mathbb{R}^{\sum d_n}) \cong \mathop{\bigtimes}\limits_{n} S^{d_n}$.
\end{theorem}

\begin{remark}
    Notice that we can also get bounds for products of wedge sums. Poincaré polynomials have the property that $\mathcal{P}(X\vee Y) = \mathcal{P}(X)+\mathcal{P}(Y)-1$. Combined with the product rule from the Künneth formula $\mathcal{P}(X\times Y) = \mathcal{P}(X)\mathcal{P}(Y)$, the Betti numbers of products of wedge sums are given by, 
    \begin{align*}
        \mathcal{P}(\mathop{\bigtimes}\limits_{r}\bigvee\limits_{n=1}^{N_r} S^{d^r_n}) &= \prod\limits_{r}( \sum\limits_{n=1}\limits^{N_r}\mathcal{P}(S^{d^r_n})-(N_r-1)) \\
        &= \prod\limits_{r}( \sum\limits_{n=1}\limits^{N_r}(1+x^{d_n^r})-(N_r-1)).
    \end{align*}

    With $\beta^r_k = \#\{n: d^r_n=k\}$ being the number of copies of $k$-dimensional spheres in the $r$-th factor, the corresponding architecture realizing such a space is $(\sum_r \max_n(d^r_n), \sum_r \sum_k \beta^{r}_k W_{k}, \sum_r \sum_k \beta^{r}_k (k+1))$
\end{remark}

\subsection{Conjectured tight and lower bounds}

So far we have only provided constructive upper bounds on representational topological expressivity. Two natural questions to ask are how tight these bounds are and whether we can also come up with lower bounds. Regarding the former, it is clear to see that the wedge sum bounds are not tight as we could do much better with the product bounds. The more interesting and tractable question is whether the bounds for the sphere construction are tight enough. The input $d$ and output layers $d+1$ are fixed so this amounts to arguing about the tightness of $W_d = \min(2d(d+1),3d+2)$. With the help of Claude Opus 5.5, we explored alternative architectures that might realize a sphere and empirically found that for $d=2,3$ there is a construction of an architecture of shape $(d,d+1,d+1)$ that does this. 

\begin{conjecture}
    Denote hidden neurons by $\{h_1(x),...,h_{d+1}(x)\}$ and output neurons by $\{p_1(x),...,p_{d+1}(x)\}$, the following equations provide a tight upper bound $(d,d+1,d+1)$ on architectures expressing the sphere $S^d$.
    
    \begin{align*}
        &h_i(x) = \max(0,x_i) \quad \mathrm{ for } \quad i\neq (d+1), \\
        & h_{d+1}(x) = \max(0,1-\sum^{d+1}_{i=1} x_i),\\
        &p_i(x) = \max(0, h_i(x)+2(1-\sum_{j=1}^{d+1} h_j(x)).
    \end{align*}

    No architecture of shape $(d,d,d+1)$ can express such a sphere. With this bound, the wedge sum bound becomes $(d,\sum_k^d \beta_k(k+1),\sum_k^d \beta_k(k+1))$ and the product bound becomes $(\sum\limits_{n=1}\limits^N d_n, \sum\limits_{n=1}\limits^N d_n + N,\sum\limits_{n=1}\limits^N d_n + N)$.
\end{conjecture}

\begin{figure*}
    \centering
    \includegraphics[width=1\linewidth]{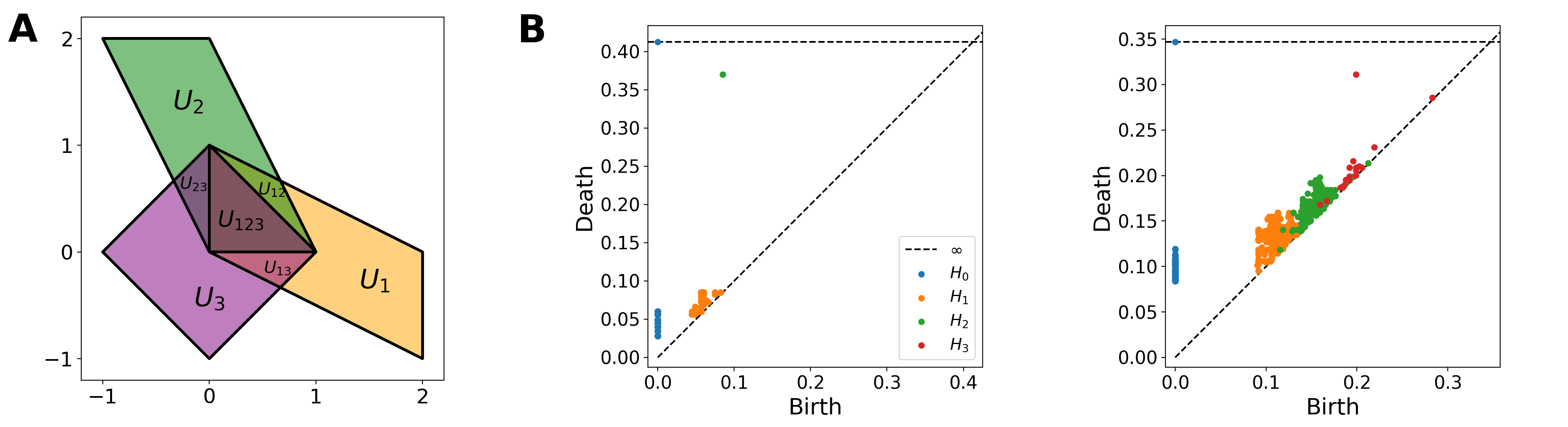}
    \caption{\textbf{A.} The active regions of the functions defining the proposed construction. \textbf{B.} The birth-death diagrams for the bottom four homology groups of $S^2$ (left) and $S^3$ (right).}
    \label{fig:tight construction}
\end{figure*}

One can see how the supports of these functions look in Figure \ref{fig:tight construction} and it is worth pointing out that their intersection pattern is exactly the same as the construction we proved in Theorem \ref{Main theorem: good cover}. While we are unable to prove this bound for spheres of arbitrary dimension in this work, we provide some empirical evidence to support its validity. For $d=2$ we applied the network to $151^2$ points from a regular mesh on $[-1,2]^2$ and for $d=3$ we used $81^3$ points on $[-1,2]^3$. We then applied the construction on the mesh and computed the persistent homology of the output using Ripser \cite{bauer_ripser_2021} with 800 landmarks. As one can see, the birth-death diagrams in panel B of Figure \ref{fig:tight construction} clearly reflect the structure of the sphere with $d=2,3$. While this is not a proof of the conjecture, it demonstrates that this construction likely works for $d<4$. We believe that the conjecture can be rigorously proven in future work.

\section{Discussion}

We have derived, what is to our knowledge, the first bounds on the representational topological expressivity of neural networks. This is in contrast to the many known results for the topological expressivity of decision boundaries. Below we discuss some consequences of our work, along with some ideas for future directions.

\subsection{Topological hierarchy for neural networks}

The bound that we derive can be easily adapted to classify networks into different complexity classes in terms of the number of neurons or parameters. For example, if we fix $d>1$ one needs $(d,3d+2,d+1)$ neurons to represent a d-sphere, then it is clearly easier to express a $d'$-sphere for $d'<d$. On the other hand, representing $S^d$ requires a total of $(3d+2)(2d+1)$ parameters (biases excluded), whereas $S^{d-1}\vee S^{d-1}$ requires $2(3d-1)^2$ parameters. The difference between them becomes $-d(12d-19)$. Therefore $S^{d-1}\vee S^{d-1}$ is easier to express while $d<2$, equally difficult to express for $d=2$ and harder to express while $d\geq 3$. Products $S^{d-1}\times S^{d-1}$ require $4(3d-1)(2d-1)$. The difference between products and wedge sums gives $2(d-1)(3d-1)>0$ implying that that products between high dimensional spheres are harder to express than wedge sums. So while it is easy to get exponential Betti numbers through products, spaces homotopy equivalent to products of spheres are harder to express than wedge sums.

We can generalize this and relate any family of wedges or products of spheres and describe how easy they are to express through the provided bounds. More research is required to gain a better understanding of the relationships between spaces and the size of the networks that can express them. However, it is important to keep in mind that these bounds are likely not optimal and if the conjectured $(d,d+1,d+1)$ bound applies the theory would be significantly simplified.

\subsection{Topological guarantees for learning}
Another way to leverage our results is as a theoretical guarantee that a particular architecture is sufficient to perform a task. This can be used when training networks to generate data with a particular topology as in \cite{bruel-gabrielsson_topology_2020} or when one wants to enforce particular geometric features in the output of a network. 

However, note that we prove that a network that can express arbitrary Betti numbers exists, not that it is consistently reachable through algorithms such as gradient descent and backpropagation. As noted in \cite{hanin2019deep} the class of functions that a network can learn from an initialization is typically smaller than the number of functions that a network can express. Therefore our result might be used as a lower bound on the architecture that can learn functions with arbitrary Betti numbers in their output. Additional work is needed to clarify the relation between learning and representational topological expressivity.

\subsection{Different input spaces}
Due to analytical tractability we have assumed that our input space is a compact contractible set in $\mathbb{R}^d$, this is often the case in practical applications of machine learning. For example, this is true if the data comes from a Gaussian distribution on $\mathbb{R}^d$ or from a uniform  distribution in some box $[-a,a]^d$. It does not apply when the data is normalized as then the input lives on a sphere. In such cases our theory would need to be adapted to that setting. 

It is also possible that the data comes from disjoint spaces, which is often the starting assumption in classification tasks. In such settings the standard notion of topological expressivity, which studies the topology of decision regions, is likely more applicable. If one is interested in non-classification tasks on such spaces, then our theory would have to also be extended to deal with $\tilde{H}_0(\Phi(X))$. We have ignored this homology group since ReLU networks are continuous PWL functions and therefore cannot cut the connected input domain and it is always true that $\tilde{H}_0(\Phi(X))=0$.

\subsection{Other invariants}
Our results rely on showing that a particular construction, realizable by a ReLU neural network, is homotopy equivalent to wedge sums and products of spheres. We then used that fact to derive a bound for the Betti numbers, which track the number of holes in a space. Homotopy equivalence is a stronger property than the Betti numbers, but the features it encodes are harder to interpret. Therefore it will be useful to consider other finer invariants of topological structure.

For example homotopy equivalence between spaces implies that they have the same cohomology ring structure. Therefore our results could be reinterpreted to classify spaces based on their cohomology ring. This is enough to differentiate the torus from the wedge space $S^2\vee S^1\vee S^1$ and our product bounds for the torus are not an improvement on the network architecture needed to express the wedge sum of spheres, but is an independent bound for tori.

Another potential angle to approach the problem of expressivity comes from tropical geometry. ReLU neural networks have been shown to be equivalent to tropical rational functions \cite{zhang2018tropical, brandenburg2024real} and there are other works connecting varieties to network function \cite{marchetti2025algebra}. This perspective implies that it might be possible to extend tools from algebraic geometry to study properties of neural representations that are not simply related to counting holes.

\section*{AI usage disclosure}
Claude Mythos and Opus (Anthropic) models have been used for literature search, text refinement and conducting the research. Specifically, we used Claude to come up with better bounds for our construction and the LLM discovered the $(d,3d+2,d+1)$ bound and proposed the conjectured tight $(d,d+1,d+1)$ bound. We also used these models to discuss and refine the proofs in the appendix.

\bibliography{example_paper}
\bibliographystyle{icml2026}

\newpage
\appendix
\onecolumn

\section{Proofs}
\label{Appendix}
Before delving into the proofs we summarize the notation we use. 
The full set of indices is given by $I=\{1,...,d+1\}$, we will denote $T=I\setminus \{d+1\}$ and use $J$ to refer to arbitrary subsets of $I$. It is convenient to write $q_j(x) = 1-\|x-a_j\|_1$ to denote the function inside the ReLU $p_j(x) = \max(0,q_j(x))$ and denote $p_j(x) = y_j$.  Coordinates in the image space are given as $y_1,...,y_{d+1}$, the set of active coordinates is given by $A(x) = \{k\in T: p_k(x)>0\}$ and its cardinality is $\alpha=|A(x)|$. The closed strata in image space are given by $\bar{R}_{J} = \{y \in \hat{P}(X) : y_j=0 \text{ } \forall j\notin J\}$. We call the set $\bar{R}_{I\setminus i} = P(\mathbb{R}^d\setminus U_i)$ the \textit{$i$-th facet} and the set $\bar{V}_I=P(\bar{U}_I)$ the \textit{core}. It is worth noting that $\bar{R}_J\cap\bar{R}_K=\bar{R}_{J\cap K}$ and $\bar{R}_J=\bigcap_{i\notin J} \bar{R}_{I\setminus i}$ which shows that intersections between the facets can produce any closed stratum. 

\begin{theorem}
    \label{Main theorem: good cover}
    The cover $C=(\bigcup\limits_{i \in I}\bar{R}_{I\setminus i})\bigcup\bar{V}_I$ is a good cover of $\hat{P}(X)$.
\end{theorem}

Every $y\in \hat{P}(X)$ either has some $y_i=0$ and is in $\bar{R}_{I\setminus i}$ or has all $y_i>0$ and is in $\bar{V}_I$. Thus $C$ is a cover. The difficult part of the proof is in showing that it is a good cover or that $\bar{R}_{I\setminus i}\cap \bar{R}_{I\setminus j}$ and $\bar{V}_I \cap \bar{R}_{I\setminus i}$ are contractible. Then by using the fact that the full intersection is empty $\bar{R}_{I\setminus 1}\cap...\cap\bar{R}_{I\setminus d+1}\cap\bar{V}_I = \emptyset$, we can use the nerve argument to conclude that the image is homotopy equivalent to a sphere.

For convenience we will split the proof of the good cover property into three lemmas proving contractibility of $\bar{R}_J$ with $d+1\notin J$ \ref{Simplex lemma}, $\bar{R}_J$ with $d+1\in J$ and $\bar{V}_I\cap \bar{R}_J$. First we introduce some more useful notation and identities.

\begin{lemma}
    \label{facts}
    Define $g(t) = |t-\tfrac{1}{2}|-|t|$, $h(t) = |t+\frac{1}{2d}|-|t|$ and $\|x\|_1=S$, then the following facts hold.
    \begin{enumerate}
        \item $q_k(x) = 1-S-g(x_k)$.
        \item $q_{d+1}(x) = 1-S-\sum_l h(x_l)$.
        \item If $S\geq \tfrac{1}{2}$, then $x_k>0$ for every $k\in A(x)$.
        \item If some $p_k(x) =0$, then $S\geq \frac{1}{2}$ and $x_k>0$ for every $k\in A(x)$.
    \end{enumerate}
\end{lemma}
\begin{proof}
    The first identity comes from the fact that we can write $q_k(x) = 1- \sum\limits_{j\neq k}|x_j|-|x_k|-|x_k-\tfrac{1}{2}|+|x_k| = 1-S-g(x_k)$.

    The second identity is similar namely, $q_{d+1}(x) = 1-\sum_l |x_l+\frac{1}{2d}| = 1+\sum_l|x_l|+\sum_l(|x_l+\frac{1}{2d}|-|x_l|) = 1-S-\sum_l h(x_l)$.

    For the third property, notice that $g(t)$ can be written in piecewise-linear form as,
    \begin{align*}
    g(t) = 
        \begin{cases}
            \tfrac{1}{2} \text{ if } t\leq 0,
            \\
            \tfrac{1}{2}-2t \text{ if } 0<t<\tfrac{1}{2},
            \\
            -\tfrac{1}{2} \text{ if } t\geq \tfrac{1}{2}.
        \end{cases}
    \end{align*}
    Then by the first identity we have $0< q_k(x) = 1-S-g(x_k)\leq \frac{1}{2}-g(x_k)$, so $g(x_k)<\tfrac{1}{2}$ which by the PWL form implies that $x_k>0$.

    The fourth property has two cases, when $k\neq d+1$, we have $p_k(x) = 1-S-g(x_k)\leq 0$ and $1-S-g(x_k)\geq \tfrac{1}{2}-S$ (the maximum of $g(t)$ is $\tfrac{1}{2}$). This implies $0\geq \tfrac{1}{2}-S$ so $S\geq \tfrac{1}{2}$. When $k=d+1$, we have $0\geq 1-S-\sum_lh(x_l)\geq\tfrac{1}{2}-S$ so the same conclusion follows. Combining this with fact 3, implies that $x_k>0$ for every $k\in A(x)$.
\end{proof}

Given these facts we now proceed to prove the first contractibility case. The goal of the proof is to establish that $\bar{R}_J$ with $d+1\notin J$ is the geometric realization of a simplex $\Delta^{\alpha}$ and is therefore contractible.
\begin{lemma}
    \label{Simplex lemma}
    (Simplex lemma) $\bar{R}_T = \Delta^d$ and $\bar{R}_J = \text{conv}\{0, e_k: k \in J\}=\{y_J\geq 0: \sum\limits_{J}y_k\leq1, \text{ } y_i=0 \text{ } i\notin J\}$. Every such $\bar{R}_J$ is contractible.
\end{lemma}

\begin{proof}
    There are two directions to prove, we begin with $\bar{R}_T\subseteq\Delta^d$. We write the sum $\sum_{k\in T} p_k = \alpha-\alpha S-\sum_{k\in A(x)} g(x_k)$.
    Fact 4 also implies that $x_k>0$ for $k\in A(x)$ and $\sum_{k\in A(x)} x_k \leq S$ is implied by $S$ being a norm. Therefore by using the inequality $g(x_k)\geq\frac{1}{2}-2x_k$ we get,
    \begin{equation*}
        \sum_{k\in T}p_k(x) \leq \alpha(1-S)-\frac{\alpha}{2}+2S=\frac{\alpha}{2} -(\alpha-2)S.
    \end{equation*}
    Then from the fact that $S(\alpha-2)\geq \tfrac{1}{2}(\alpha-2)$ we get $\sum_{k\in T} p_k(x) \leq\tfrac{\alpha}{2}-(\alpha-2)S\leq \tfrac{\alpha}{2}-\tfrac{(\alpha-2)}{2}=1$ which implies that the image is in the simplex. The bound is trivial when $\alpha\leq 1$.

    For the second direction $\Delta^d\subseteq\bar{R}_T$, let $y\in \Delta^d$ and $\sum_k y_k = \|y_T\|_1=S_P$ because all coordinates of the simplex are non-negative (we also have that $S_P\leq 1$). For each $y$ we have to pick a region in which $P$ is affine and then find an $x$ that realizes that $y$ while keeping $p_{d+1}(x)=y_{d+1}=0$. 

    For $d\geq 3$ take $x_k\in[0,\tfrac{1}{2}]$, from facts 1 and 2 along with the expansions of $g(x_k)$ and $h(x_k)$ each equation can be written as,
    \begin{align*}
        &q_k(x) = \frac{1}{2}-S+2x_k,
        \\
        &p_{d+1}(x) = \max(0,1-S - \frac{d}{2d})= \max(0,\frac{1}{2}-S)=0.
    \end{align*}

    From solving these equations we get $x_k =\tfrac{1}{2}(y_k+S-\tfrac{1}{2})$ and since $\sum x_k = S\geq\tfrac{1}{2}$ (fact 4 and $p_{d+1}(x)=0$), we can write $2S = S_P+(S-\tfrac{1}{2})d$. After rearranging, we get $S= \frac{(d/2)-S_P}{d-2}$ and therefore,
    \begin{equation*}
        x_k = \frac{1}{2}(y_k+\frac{d/2-S_P}{d-2}-\frac{1}{2}).
    \end{equation*}
    It remains to show that these solutions are in fact in the interval $[0,\tfrac{1}{2}]$. $x_k\geq 0$ follows from $S\geq\tfrac{1}{2}$. For $x_k\leq \tfrac{1}{2}$, we have $y_k +\frac{d/2-S_P}{d-2}\leq\frac{3}{2}$. The worst case scenario is when $y_k=S_P$ and the expression becomes $S_P(d-2)+\frac{d}{2}-S_P\leq\frac{3}{2}d-3$. After simplifying we get $S_P(d-3)\leq (d-3)$ and therefore $S_P\leq 1$ when $d>3$ (for $d=3$ we get 0). Therefore, for any $y=(y_T,0)$ we can construct an $x$ such that $P(x)=y$.

    For $d=2$ we can explicitly guess the solution $x^*=(1-\frac{S_P}{2},\tfrac{1}{2}-\tfrac{y_1-y_2}{2})$ when $y_1\geq y_2$. Then we get $x_1=1-\frac{S_P}{2}\geq\frac{1}{2}$ and thus $g(x_1) = -\tfrac{1}{2}$. From $y_1-y_2\in[0,1]$ we get $x_2=\frac{1}{2}-\frac{y_1-y_2}{2}\in[0,\frac{1}{2}]$ and $g(x_2) = \frac{1}{2}-2x_2$. Then the equations for $P$ become,

    \begin{align*}
        &p_1(x) = \frac{3}{2}-S \\
        &p_2(x) = \frac{1}{2}-S+2x_2\\
        &p_3(x) = \max(0,1-S) = 0.
    \end{align*}
    Then by symmetry, when $y_2>y_1$, the solutions are given by switching the arguments in $x^*$. 

    For $d=1$ we have $\bar{R}_T=[0,1]$.
    
    Given that we have established both $\bar{R}_T\subseteq \Delta^d$ and $\Delta^d \subseteq\bar{R}_T$, we conclude that $\bar{R}_T=\Delta^d$. Furthermore, we have that $\bar{R}_J=\bar{R}_T\cap \{y_k=0, k\in T\setminus J\}$ are faces of $\Delta^d$. Therefore, any such set is contractible.
\end{proof}

Next we prove the case in which $d+1\in J$, it turns out that in this case $\bar{R}_J$ is not a simplex so the proof will proceed by constructing a strong deformation retraction to $\bar{R}_{J\setminus d+1}$ which is a simplex by Lemma \ref{Simplex lemma}. Before jumping into the proof, we state a useful definition for labeling the different regions of the domain.

\begin{definition}
    The sign vector at a point $x$ is given by a map $\tau:\mathbb{R}^d\to \{W^-,W^+,E^-,E^+\}^d$ with components,

    \begin{align*}
        \tau(x)_j = 
        \begin{cases}
            W^- \text{ if } x_j<-\frac{1}{2d},
            \\
            W^+ \text{ if } -\frac{1}{2d}\leq x_j<0,
            \\
            E^- \text{ if } 0 \leq x_j < \frac{1}{2},
            \\
            E^+ \text{ if } \frac{1}{2}\leq x_j.
        \end{cases}
    \end{align*}
    Then the full sign vector is $\tau(x) \in \{W^-, W^+,E^-,E^+\}^d$, we also define $W=\{W^-, W^+\}$ and $E=\{E^-,E^+\}$. By abuse of notation we will, for example, write $t\in E^-$ to mean $t\in [0,\frac{1}{2})$.
\end{definition}

\begin{lemma}
\label{Lowering lemma}
    (Lowering lemma) Given $d+1\in J\subsetneq I$, the simplex $\bar{R}_{J\setminus \{d+1\}}$ (Lemma \ref{Simplex lemma}) is a strong deformation retract of $\bar{R}_J$. Therefore $\bar{R}_J$ is contractible.
\end{lemma}

\begin{proof}
    Let $y=(y_T,y_{d+1})\in \bar{R}_J$ with $y_{d+1}>0$. Our goal will be to show that for every $y_{d+1}$ the set $\mathcal{B} = \{c\in[0,y_{d+1}]: (y_T,c) \in \hat{P}(X)\}$ is a lower set of $[0,y_{d+1}]$ which implies that $\mathcal{B} = [0,y_{d+1}]$ and the straight-line homotopy $H(y,t)=(y_T,(1-t)y_{d+1})$ maps into $\bar{R}_J$. Such a homotopy has the properties of a strong deformation retract. 
    \begin{enumerate}
        \item $H(y,t)$ is continuous, as can be seen from its formula.
        \item $H(y,t)\in \bar{R}_J$, since the zero pattern outside $J$ never changes and $(1-t)y_{d+1}\in[0,y_{d+1}]=\mathcal{B}$.
        \item It is a strong deformation retraction because $H(y,0) = y$, $H(y,1) = (y_T,0)\in \bar{R}_{J\setminus \{d+1\}}$ and $H((y_T,0),t)=(y_T,0)$ for all $t$ (this is the "strong" condition).
    \end{enumerate}

    We again have to split the proof into two cases depending on how many coordinates of $y_T$ are positive.
    
    \textbf{Case 1: $2\leq \alpha\leq d-1$.} Let $0<c\in \mathcal{B}$ and $x$ be some preimage of $(y_T,c)$ so that $p_{d+1}(x)=c >0$. By facts 3 and 4 from Lemma \ref{facts}, the active coordinates $k\in A(x)$ are positive or $x_k\in E$. Furthermore, if $x_l\geq\frac{1}{2}$ for some $l\in A(x)$ and $l\neq k \in A(x)$ (such a $k$ exists when $\alpha\geq2$), then we can write the equations $q_k(x) = 1 - \|x-a_k\|_1$ and  $q_{d+1}(x) = 1 - \|x-a_{d+1}\|_1$. Adding the two gives,
    \begin{equation*}
        q_{d+1}(x)+q_k(x) = 2-(\|x-a_k\|_1+\|x-a_{d+1}\|_1).
    \end{equation*}
    The sum of norms can be split by the index-ordered pairs. For the $l$-th entry we get $|x_l|+|x_l+\frac{1}{2d}|=2x_l+\frac{1}{2d}$. For the others we can use the triangle inequality to write $|x_j-a_{k,j}|+|x_j-a_{d+1,j}| \geq |a_{k,j} - a_{d+1,j}|$. Summing all of these terms together gives 
    \begin{equation*}
        q_{d+1}(x)+q_k(x) \leq 2 - (1-\frac{1}{2d}+2x_l+\frac{1}{2d}) = 1-2x_l\leq 0. 
    \end{equation*}
    Therefore the assumption that $x_l\geq\frac{1}{2}$ implies that at least one of the two positive coordinates are negative. This is a contradiction, so $x_l < \frac{1}{2}$ and $\tau(x)_l=E^-$ have to hold. This allows us to write an affine expression for all active coordinates $k\in A(x)$,
    \begin{equation*}
        q_k(x) = \frac{1}{2}-S + 2x_k.
    \end{equation*}
    Notice that the flow we wish to define keeps every coordinate fixed except $y_{d+1}$. Therefore, the active coordinates should not change, the inactive ones can change as long as they don't become active and the $d+1$ coordinate has to move so that $y_{d+1}$ decreases. We take a single inactive coordinate with $\tau(x)_{j_0}\in W$, which exists because otherwise we get $x>0$ and $p_{d+1}(x) = 1- \sum(x_l+\frac{1}{2d}) = 1- S-\frac{1}{2} = \frac{1}{2}-S$. From fact 4, this contradicts $p_{d+1}(x) > 0$. A set of flow equations satisfying the conditions above is given by,

    \begin{align}
        &\dot{x}_k = \frac{1}{2}, \quad \dot{p}_k = 1-\dot{S} \text{ for } k\in A(x),\\
        &\dot{x}_{j_0} =  \frac{\alpha-2}{2}, \\
        &\dot{x}_j = 0 \text{ for } j\in T\setminus A(x), j\neq j_0.
    \end{align}
    Observe what this does to the derivative $\dot{S} = \sum_{k\in T}\text{sign}(x_l)\dot{x}_k=\frac{\alpha}{2}-\frac{\alpha-2}{2} = 1$, where we used $x_k>0$ for $k\in A(x)$, $x_{j_0}<0$ and $\dot{x}_l=0$ for $l\neq j_0$ and $l\notin A(x)$. This implies that $\dot{p}_k = 0$ for $k\in A(x)$. Additionally for $l\notin A(x)$, we have $\dot{q}_{l} = -\dot{S}=-1$ since $g(x_{j_0}) = \frac{1}{2}$, which means that $p_{l}$ including $p_{j_0}$ stays $0$ for small $t$. What remains to be shown is that $p_{d+1}(x)$ decreases along the flow. There are two cases to separate depending on whether $\tau(x)_{j_0} =W^-$ or $W^+$. 

    By applying fact 2 we have $\dot{p}_{d+1} = -\dot{S}-\sum_{l\in T}\dot{h}(x_l)$. One can see that the only contributing term is $j_0$ from the derivatives of $h(x_l)$,

    \begin{align*}
        \frac{\partial h(x_l)}{\partial x_l} = 
        \begin{cases}
            0 \text{ for } \tau(x)_l = W^-,
            \\
            2 \text{ for } \tau(x)_l = W^+,
            \\
            0 \text{ for } \tau(x)_l \in E.
        \end{cases}
    \end{align*}
    Then if $\tau(x)_{j_0} =W^-$, we have $\dot{h}(x_{j_0})=\frac{\partial h(x_{j_0})}{\partial x_{j_0}}\frac{\partial x_{j_0}}{\partial t} = 0$ and $\dot{p}_{d+1} = -\dot{S}=-1$. If $\tau(x)_{j_0} = W^+$, then we have $\dot{h}(x_{j_0})=2\dot{x}_{j_0}$ and $\dot{p}_{d+1} = -\dot{S}-2\dot{x}_{j_0} = 1-\alpha$. In both cases this derivative is negative so $y_{d+1}=p_{d+1}$ decreases along the flow. Using this property we can now prove that $\mathcal{B} = [0,y_{d+1}]$.

    First off, $\hat{P}(x)$ vanishes outside of $\Omega$ so it is compact, which implies that $\mathcal{B}$ is a closed subset. Let $m^* = \inf\{c: [c,y_{d+1}]\subseteq \mathcal{B}\}$, by the closed property we have $[m^*,y_{d+1}]\subseteq \mathcal{B}$. Assume that $m^*>0$ and apply the flow at a preimage of $(y_T,m^*)$ which moves the point to $(y_T,m^*-\delta)$ and therefore $[m^*-\delta,y_{d+1}]\subseteq \mathcal{B}$.  This is a contradiction, so $m^*=0$ and $\mathcal{B}=[0,y_{d+1}]$ and $H(y,t) = (y_T,(1-t)y_{d+1})$ maps into $\bar{R}_J$ and is a strong deformation retract onto the simplex $\bar{R}_{J\setminus\{d+1\}}$.

    \textbf{Case 2:} $\alpha\leq 1$. Using the previous flow equations does not work for $\alpha=1$ as then we cannot derive that active coordinates have $x_k\leq\frac{1}{2}$ and get $\dot{p}_{d+1} = 0$ in the $W^+$ regime. This setting only has two equations to keep track of so we opt for a more explicit approach where we show that the set $\bar{R}_{1,d+1}\setminus \bar{R}_{d+1} = \{y_1e_1+y_{d+1}e_{d+1}:y_1>0 \text{ and } 0\leq y_{d+1}\leq \min(1-y_1,1-\tfrac{1}{d})\}$ is a lower set with respect to $y_{d+1}$. For the bounds, take $p_1(x)>0$, by the triangle inequality we have the sum, 
    \begin{align*}
        q_1(x)+q_{d+1}(x) &= 2 - (\|x-a_1\|_1+\|x-a_{d+1}\|_1) \\
        &\leq 2-\|a_1-a_{d+1}\|_1=1,
    \end{align*}

    hence $p_{d+1}(x) \leq1-p_1(x)$. Then using facts 2 and 4 we can also write $q_{d+1}(x) = 1-S - \sum h(x_k) = 1-S-\tfrac{1}{2d}-\sum_{l\neq 1}h(x_l)$. From the expression of $h$ we have $h(x_l)\geq -\tfrac{1}{2d}$ and $S\geq \frac{1}{2}$ we have, 
    \begin{align*}
        q_{d+1}(x) &\leq 1-S-\tfrac{1}{2d}+\tfrac{d-1}{2d}
        \\
        &\leq \tfrac{1}{2}+\tfrac{d-2}{2d}=1-\tfrac{1}{d},
    \end{align*}
    therefore $p_{d+1}(x) = \max(0,q_{d+1})\leq 1- \tfrac{1}{d}$
    Then to realize the points satisfying both bounds we take the same approach as in Lemma \ref{Simplex lemma}. Specifically, we solve $(p_1,p_{d+1})$ for a point $x=(x_1,s,...,s)$. When $\tau(x)_1 = E^-$ and $s\in W^-$, we get the solutions,
    \begin{align*}
        x_1 = \frac{1}{2}(y_1-y_{d+1}+1-\tfrac{1}{d}), && s = -\frac{2-\tfrac{1}{d}-y_1-y_{d+1}}{2(d-1)}
    \end{align*}
    One can check these solutions apply when $y_{d+1}\geq y_1-\tfrac{1}{d}$. Otherwise, we consider $\tau(x)_1 = E^+$ and $s\in W^+$ which gives the solutions,
    \begin{align*}
        x_1 = 1-\frac{1}{2}(y_1+y_{d+1}), && s = -\frac{1-y_1+y_{d+1}}{2(d-1)}.
    \end{align*}
    As one can check these solutions are consistent with $y_{d+1} \leq y_1-\frac{1}{d}$. This argument applies for any $\bar{R}_{k,d+1}$ due to the symmetry of the construction.

    Finally, for $\alpha=0$ we have $\bar{R}_{d+1} = \{y_{d+1}e_{d+1}: 0\leq y_{d+1} \leq 1\}$. This is a lower set wrt. $y_{d+1}$ realized by $x = -\lambda\sum e_l$ with $\lambda \in [\frac{1}{2d},\frac{1}{2d}+\frac{1}{d}]$. Every $q_k(x) = \frac{1}{2}-d\lambda\leq 0$ and $p_{d+1}(x) = 1- d(\lambda-\frac{1}{2d})$ where at $\lambda=\frac{1}{2d}$ we have $p_{d+1}(x) = 1$ and at $\lambda = \frac{1}{2d}+\frac{1}{d}$, we have $p_{d+1}(x)=0$. Since we know that $p_{d+1}(x)\leq 1$ this completes the proof.

    Every $y\in\bar R_J$ has $0\leq\alpha\leq d-1$, so by the cases above
$\mathcal S=[0,y_{d+1}]$ for every $y\in\bar R_J$, i.e.\ $\bar R_J$ is a lower set with
respect to $y_{d+1}$. Hence items~1--3 hold, and $H$ is a strong deformation retraction of
$\bar R_J$ onto $\bar R_{J\setminus\{d+1\}}$.

    To conclude, for any $y\in \bar{R}_J$, we have $\mathcal{B}=[0,y_{d+1}]$ for all cases $0\leq\alpha\leq d-1$. This implies $\bar{R}_{J}$ is a lower set wrt. $y_{d+1}$ and all three properties of $H(y,t)$ hold. This implies that $H(y,t)$ is a strong deformation retract onto the simplex $\bar{R}_{J\setminus \{d+1\}}$.
\end{proof}

The last lemma needed to establish the good cover property of $C$ is the contractibility of the core $\bar{V}_I$ and its intersections $\bar{V}_I\cap \bar{R}_J$. We split this into two lemmas.

\begin{lemma}
    \label{Contractible core lemma}
    (Contractible core lemma) The core $\bar{V}_I=\hat{P}(\bar{U}_I)$ is contractible.
\end{lemma}

\begin{proof}

    \textbf{Step 1:} $\bar{U}_I=\{S\leq \frac{1}{2}\}$.
    
    Consider the open set $U_I$ on which all functions $p_k$ are active. We have $U_I=\{S< \frac{1}{2}\}$. If $S< \frac{1}{2}$, then by the triangle inequality we have $\|x-a_k\|_1< S+\|a_k\|_1< 1$, which implies $p_k(x) = 1- \|x-a_k\|>0$. Since $\bar{U}_I$ is the closure of $U_I$, we get $\bar{U}_I=\{S\leq\frac{1}{2}\}$. 
    Conversely, suppose that all $p_k(x)>0$ and $S\geq \frac{1}{2}$. Then from fact 3 we get $x_k>0$. This implies that $h(x_k)=\frac{1}{2d}$, so $p_{d+1}(x) < \frac{1}{2}-\sum h(x_k) = 0$. This is a contradiction so $\bar{U} = \{S\leq \frac{1}{2}\}$. 
    
    It is convenient to rewrite the equations again given that $x_k\leq \frac{1}{2}$ (which follows from $S\leq \frac{1}{2}$). \footnote{Notice that since all functions are active we can freely write $p_k(x) = q_k(x)$.} Define $t^\pm=\max(\pm t,0)$, then we get the expressions,

    \begin{align}
    \label{eq: other formulation}
        p_k(x) = \frac{1}{2}-S+2x^+_k \quad (k\in T), \quad p_{d+1}(x) = \frac{1}{2}-S + 2\sum \min(x_l^-,\frac{1}{2d}).
    \end{align}

    \textbf{Step 2:} $\bar{U}_I$ contracts to $a_{d+1}$ and this factors through $\hat{P}$.
    
    We now proceed to prove the contractibility of $\bar{V}_I$ by first constructing a contraction $x(t)$ in the domain $\bar{U}_I$ and then showing that $(x,t)\to \hat{P}(x(t))$ factors through the image map $\hat{P}\times \mathrm{id}$, which is equivalent to requiring that $\hat{P}(x) = \hat{P}(x')$ implies $\hat{P}(x(t)) = \hat{P}(x'(t))$.

    Take the contraction to $a_{d+1}$ given by $x(t) = (1-t)x+ta_{d+1}$. This is a contraction of $\bar{U}_I$ since $a_{d+1}\in \bar{U}_I$, that set is convex and makes $p_{d+1}(x(t))$ linear in $t$ from the equality,
    \begin{align*}
        p_{d+1}(x(t)) &= 1- \|(1-t)x+ta_{d+1}-a_{d+1}\|_1 \\
        &=1-(1-t)\|x-a_{d+1}\|_1\\
        &=1-(1-t)(1-y_{d+1})=p_{d+1}(x'(t)).
    \end{align*}

    This property does not hold for the other functions $p_k(x)=y_k$, so instead we show that they are determined by $y$. Use the equality $\hat{P}(x) = \hat{P}(x')$ and equation \ref{eq: other formulation} to express,
    \begin{align*}
        p_k(x) &= \frac{1}{2}-\|x\|_1+2x_k^+\\
        &=\frac{1}{2}-\|x'\|_1+2{x'}_k^+=p_k(x')
    \end{align*}
    By rearranging the individual terms we get $p_k(x)-(\frac{1}{2}-S)=2x_k^+\geq 0$ and $p_{d+1}(x) -(\frac{1}{2}-S) = 2\sum \min(x^-_l,\frac{1}{2d})\geq 0$. This implies that at least one $p_i(x) = \frac{1}{2}-S$, since either $x_k\leq0$ for some $k$ and thus $p_k(x)-(\frac{1}{2}-S) = 0$ implying $p_k(x) = \frac{1}{2}-S$ or $x>0$ implying $x_l^-=0$ and $p_{d+1}(x) = \frac{1}{2}-S$ and $\min_i(y_i)\leq \frac{1}{2}-S$. We also have $\min_i(y_i)\geq \frac{1}{2}-S$ from the fact that the formulas in \ref{eq: other formulation} add a positive term to $\frac{1}{2}-S$. Therefore, $S = \frac{1}{2}-\min_i(y_i)$ only depends on $y$, so $S=\|x\|_1=\|x'\|_1$ and $x_k^+={x'}_k^+$.

    It remains to show that $\|x(t)\|_1=\|x'(t)\|_1$ and $x(t)^+=x'(t)^+$. We can define the set where $x$ and $x'$ have non-positive values as $Z = \{l:x_l\leq 0\} = \{l: {x'}_l\leq 0\}$. Off this set we have $x_l(t)=x'_l(t)$. On this set we have, $\max((1-t)x-\frac{t}{2d},0)=x(t)^+=x'(t)^+=0$ and $|x_l(t)| = (1-t)|x_l|+\frac{t}{2d}$. In other words, the different preimages of $y$ only differ only in coordinates in $Z$. Then summing over $N$ gives $\sum_{l\in Z}|x_l(t)| = (1-t)\sum_{l\in Z}|x_l|+\frac{t|Z|}{2d} = \sum_{l\in Z}|x_l'(t)|$, with the last equality coming from the fact that $Z$ is the same for both points. Thus, we get that off $Z$ the coordinates agree and on $Z$ the sums of the absolute values agree, this implies $\|x(t)\|_1=\|x'(t)\|_1$ and $x(t)^+=x'(t)^+$. From this we see that $\hat{P}(x) = \hat{P}(x')$ implies $\hat{P}(x(t)) = \hat{P}(x'(t))$.

    \textbf{Step 3:} $\bar{V}_I$ is contractible.
    
    From the previous step we concluded that $\hat{P}(x(t))= y(t)$ is independent of the choice of the starting point $x(0)\in \bar{U}_I$ with $P(x(0))=y$. Another way to state this property is to say that the map $F(x,t)=\hat{P}(x(t))$ factors through $\hat{P}\times \mathrm{id}$, meaning that the following diagram commutes,
    \[
    \begin{tikzcd}
      \bar U_I\times[0,1] \arrow[r, "F"] \arrow[d, "\hat{P}\times\mathrm{id}"'] & \bar V_I \\
      \bar V_I\times[0,1] \arrow[ur, dashed, "\Phi"']
    \end{tikzcd}
    \]
    $\hat{P}$ is continuous, $\bar{U}\times[0,1]$ is compact and $P\times \mathrm{id}$ is a continuous surjection (simply because $\bar{V}_I=\hat{P}(\bar{U}_I)$) onto a Hausdorff space. Thus it is closed and a quotient map. $F(x,t) = \hat{P}(x(t))= \Phi\circ P\times\mathrm{id}$ is also continuous due to the continuity of $\hat{P}$, so $\Phi$ is continuous. This allows us to write the retraction on the image space $\bar{V}\times [0,1]$ as $\Phi(\hat{P}(x),t)$. We then have $\Phi(\hat{P}(x),0) = \hat{P}(x)$ because $x(0)=x$ and $\Phi(\hat{P}(x),1) = e_{d+1}$ because $x(1) = a_{d+1}$ and $\hat{P}(a_{d+1}) = e_{d+1}$. Thus $\Phi$ is a homotopy starting from the identity and ending at the constant map with value $e_{d+1}$. Therefore the identity is null-homotopic and $\bar{V}_I$ is contractible.
\end{proof}

The final step to establishing that $C$ is a good cover is to confirm that intersections of the type $\bar{R}_J\cap\bar{V}_I$ are contractible.

\begin{lemma}
    \label{Contractible intersection lemma}
    (Contractible intersection lemma) For $\emptyset \neq J\subsetneq I$, the intersections $\bar{R}_J\cap\bar{V}_I$ are contractible.
\end{lemma}

\begin{proof}
    It is again easier if we split the proof depending on whether $d+1 \in J$.

    \textbf{Step 1:} $d+1\notin J$.

    On $\bar{V}_I$ we have $S\leq \frac{1}{2}$ and  on $\bar{R}_J$, with $J=T$, we have $S\geq \frac{1}{2}$. Therefore, on their intersection we have $S=\frac{1}{2}$ and $p_{d+1}(x)=0$. In addition, fact 3 implies that $S = \sum x_k=\frac{1}{2}$ and we can write the equations on this intersection as,

    \begin{align*}
        p_k(x) = 2x_k, \quad p_{d+1}(x) = 0.
    \end{align*}

    From here we see that $\sum p_k(x) = 2\sum x_k$, which by the fact that $\sum x_k =\frac{1}{2}$ implies that $\sum_k p_k(x) = 1$. Furthermore, $p_k(a_k) = e_k$ and $p_l(a_k) = 0$ when $l\neq k$, meaning that the image under $\hat{P}$ is the simplex. Choosing $J\subset T$, simply restricts the simplex to some lower-dimensional face. In any case these intersections are convex and contractible simplices.

    \textbf{Step 2:} $d+1\in J$.

    Denote the base of $\bar{R}_J\cap\bar{V}_I$ , excluding $p_{d+1}$, by $\Delta_J = \{y_k\geq 0: \sum\limits_{k\in J\setminus \{d+1\}} y_k\leq 1\}$. $\Gamma \subset \bar{R}_J\cap\bar{V}_I$ is a graph over this base in which, for $y\in \Delta_J$, points in the intersection are given by $(y,y_{d+1})$. Our goal is to show that $\bar{R}_J\cap\bar{V}_I$ deformation retracts to such a graph $\Gamma\simeq \Delta_J$.  We can write the expression for $p_{d+1}(x) = y_{d+1}$ as,

    \begin{align*}
        p_{d+1}(x) &= \frac{1}{2}-\sum h(x_l) \\
        &= \frac{1}{2}-\frac{d}{2d} + 2\sum\limits \min(x_l^-,\tfrac{1}{2d}) \\
        &= 2\sum \min(x_l^-,\tfrac{1}{2d}).
    \end{align*}

    We can decompose every $x\in \Sigma_J = \{x: x_l\leq 0 \text{ for } l\notin J, \|x\|_1 = \frac{1}{2}\}$ into its positive and negative parts. Then for some index $j_0\in T\setminus J$ (which exists since $J\neq I$), we can write a new point as $\hat{x} = x^+-\|x^-\|_1 e_{j_0}$. Since the $l_1$ norm splits componentwise, the norm of this new point is still $\|\hat{x}\|_1 = \|x^+\|_1+\|x^-\|_1 = \frac{1}{2}$ which means that $\hat{x} \in \Sigma_J$. The positive parts of $x$ and $\hat{x}$ are the same so we get $p_k(x) = p_k(\hat{x}) = 2x_k$. For $p_{d+1}$ the positive parts contribute $0$ and we have $p_{d+1}(\hat{x})= 2\min(\|x^-\|_1,\frac{1}{2d})$. The positive sum $\|x^+\|_1$ can also be written as $\frac{1}{2}\sum y_l=\|x^+\|_1$, which gives $\|x^-\|_1 = \frac{1}{2}-\|x^+\|_1=\frac{1}{2}(1-\sum y_l)$. So $p_{d+1}(\hat{x})$ can be realized through the image coordinates $y$ as $p_{d+1}(\hat{x}) = \min(1-\sum y_l,\frac{1}{d})$. This last expression will be useful later so we use a new function $z^*(y) = \min(1-\sum y_l,\frac{1}{d})$ and define $\Gamma$ as,

    \begin{equation*}
        \Gamma = \{(y,z^*(y)):y\in \Delta_J\}.
    \end{equation*}
    
    Since they share signs, $x$ and $\hat{x}$ are in the same closed orthant on which $\|x\|_1$ is linear. Therefore, the segment between them stays in a convex set . Let this segment be parametrized as $x(s) = (1-s)x+s\hat{x}$. The image of the positive maps on this path is constant by,
    
    \begin{align*}
        p_k(x(s)) &= 2x_k^+(s) = 2\max((1-s)x_k+s\hat{x}_k,0)\\
        &=2(1-s)x_k+2sx_k = 2(1-s)\hat{x}_k+2s\hat{x}_k \\&= 2\hat{x}{_k} = p_k(\hat{x}(s)).
    \end{align*}
    The remaining negative coordinates map to $0$ and for $p_{d+1}$ we have a continuous function,
    \begin{align*}
        p_{d+1}(x(s)) = 2\sum \min(x_l(s)^-,\tfrac{1}{2d}).
    \end{align*}

    By substitution, we can check that $p_{d+1}(x(0)) = p_{d+1}(x)$ and $p_{d+1}(x(1)) = z^*(y)$. Given the continuity of $p_{d+1}(x(s))$, any value in the interval $[z^*(y),p_{d+1}(x)]$ is realized for some $s$. By the fact that there is at least one negative coordinate we have $0<2\min(x_l^-,\frac{1}{2d})$ for some $l$, which implies that both extremes of this interval are positive. Therefore, $\hat{P}(x(s))$ stays in $\bar{R}_J\cap\bar{V}_I$.

    Finally, we can define a homotopy $G((y,z),s) = (y,(1-s)z+sz^*(y))$ on the image. It is continuous, $G((y,z),0) = (y,z)$ is the identity and since $\Delta_J$ is fixed throughout, $G((y,z),1)= (y,z^*(y))$ maps into a graph $\Gamma$ over $\Delta_J$. When restricted to $z^*(y)\in\Gamma$, the homotopy keeps $\Gamma$ fixed, since $G((y,z^*(y)),s) = (y,(1-s)z^*(y)+sz^*(y))=(y,z^*(y))$. Since the base is a convex set, hence contractible, it is enough to establish that  $\Gamma\simeq \Delta_J$. This can be done through the projection $(y,z)\to y$, which is continuous and has a continuous inverse given by $y\to (y,z^*(y))$. Therefore, the homotopy $G$ is a deformation retract of $\bar{V}_I\cap \bar{R}_J$ to $\Gamma$ which is homeomorphic to a convex set $\Delta_J$ and is therefore contractible. This concludes the proof.
\end{proof}

With this last lemma, we have now checked all possible intersections of the cover $C=(\bigcup\limits_{i \in I}\bar{R}_{I\setminus i})\bigcup\bar{V}_I$ and we can now prove Theorem \ref{Main theorem: good cover}.

\begin{proof}
    The proof consists in showing that every intersection of the cover is contractible, which we have done in the previous four lemmas. Here we simply list them.

    \begin{enumerate}
        \item Intersections of the type $\bar{R}_J$ with $d+1\notin J$ are simplices by Lemma \ref{Simplex lemma}.
        \item Intersections of the type $\bar{R}_J$ with $d+1\in J\neq I$ are (strong) deformation retract to simplices by Lemma \ref{Lowering lemma}.
        \item The core $\bar{V}_I = \hat{P}(\bar{U}_I)$ is contractible by Lemma \ref{Contractible core lemma}.
        \item Intersections of the type $\bar{V}_I\cap\bar{R}_J$, with $\emptyset \neq J\subsetneq I$ contract to simplices when $d+1\notin J$ or to a graph over simplices when $d+1\in J$ by Lemma \ref{Contractible intersection lemma}.
    \end{enumerate}
    Given that all this list exhausts all possible intersections, $C$ is a good cover of $\hat{P}(X)$.
\end{proof}

The fact that $C$ is a good cover allows us to apply the nerve theorem, which leads to our final theorem.

\begin{theorem}
    $N(C)$ is homotopy equivalent to the sphere $S^d$.
\end{theorem}
\begin{proof}
    This simply follows by looking at the intersection pattern of all pieces of the cover. The conditions $\bar{R}_K\cap\bar{R}_J = \bar{R}_{K\cap J}$ along with the intersections $\bar{V}_I\cap \bar{R}_J \neq \emptyset$ imply that the nerve $N(C)$ contains every proper subset of the $d+2$ possible vertices. The only special intersections are the one including all closed strata $\bigcap (R_{I\setminus i}) = \textbf{0}$ and its intersection with the core $\textbf{0}\cap \bar{V}_I = \emptyset$. Therefore, the nerve has $d+2$ vertices which realize all intersections except for the $d+1$ face. This is equivalent to taking the boundary $\partial\Delta^{d+1}$, which is homotopy equivalent to $S^d$.
\end{proof}


\end{document}